\documentclass{amsart}
\usepackage[utf8]{inputenc}
\usepackage{amsmath, amssymb, amsthm, amsfonts, mathtools, array, xfrac, float, multirow, indentfirst, url, enumerate, enumitem, comment, afterpage, pifont, makecell, xcolor, hyperref, bbm}
\usepackage[margin=1.5in]{geometry}

\usepackage[backend=bibtex, sorting=none, bibstyle=numeric, citestyle=numeric, sorting=nyt,maxbibnames=99, giveninits=false, isbn=false, url=false, doi=false, eprint=false]{biblatex}
\usepackage{comment}
\usepackage{thmtools}
\usepackage{cleveref}
\allowdisplaybreaks

\newtheorem{theorem}{Theorem}
\newtheorem{lemma}[theorem]{Lemma}

\newtheorem{corollary}{Corollary}
\newtheorem{example}{Example}
\newtheorem{remark}{Remark}

\renewcommand{\Re}{\mathrm{Re}}
\renewcommand{\Im}{\mathrm{Im}}

\theoremstyle{definition}
\newcommand{\ov}{\overline}

\newcommand{\bb}{\mathbb}
\newcommand{\mr}{\mathrm}
\newcommand{\md}{\,\mathrm{d}}

\makeatletter
\declaretheoremstyle
    [headformat={\NOTE}, % only note as label
    notebraces={}{}, % removing braces from note
    notefont=\bfseries, % making the note bold
    preheadhook=\def\thmt@space{}, % removing space before the note
    numbered=no
    ]{namedtheorem}
\makeatother

\begin{document}

\author{Tianyu Zhao}

\title{Upper bounds on gaps between zeros of $L$-functions}

\email{zhao.3709@buckeyemail.osu.edu}
\address{Department of Mathematics, The Ohio State University, 231 West 18th
Ave, Columbus, OH 43210, USA.}


\begin{abstract}
    We prove two unconditional upper bounds on the gaps between ordinates of consecutive non-trivial zeros of a general $L$-function $L(s)$. This extends previous work of Hall and Hayman (2000) on the Riemann zeta-function and work of Siegel (1945) on Dirichlet $L$-functions. Interestingly, we observe that while Hall and Hayman's method gives a sharper estimate when the degree of $L(s)$ is sufficiently small compared to the analytic conductor, Siegel's method does better in the other regime. 
\end{abstract}

\maketitle

\section{Introduction}

\subsection{Background}
One important theme in analytic number theory is the investigation of spacings between zeros of $L$-functions. The prototypical example of an $L$-function is the Riemann zeta-function, defined for $\Re(s)>1$ by the Dirichlet series
\[
\zeta(s)=\sum_{n=1}^\infty \frac{1}{n^s},
\]
which can be analytically continued to $\bb{C}$ except for a simple pole at $s=1$. The classical Riemann\textendash von Mangoldt formula (see, e.g., \cite[Ch.15]{Dav}) gives an estimate for $N(T)$, the number of non-trivial zeros $\rho=\beta+i\gamma$ of $\zeta(s)$ with $0<\gamma\leq T$:
\begin{equation}\label{Riemann-von mangoldt}
    N(T)=\frac{T}{2\pi}\log \frac{T}{2\pi e}+S(T)+O(1), \qquad S(T)=O(\log (T+3)).
\end{equation}
If we order the positive ordinates of zeta zeros by $0<\gamma_1\leq \gamma_2\leq \ldots$, this formula implies that the gaps $\gamma_{n+1}-\gamma_n$ are bounded by an absolute constant. In 1924, Littlewood \cite{Lit24a} proved a qualitatively better bound as $n\to \infty$, namely, there exists a positive constant $A>0$ (he showed that $A=32$ is admissible) such that
\begin{equation}\label{Littlewood bound}
    \gamma_{n+1}-\gamma_n \leq \frac{A}{\log\log\log \gamma_n}
\end{equation}
for all sufficiently large $n$. The proof uses the theory of elliptic functions and conformal mapping in complex analysis (see \cite[\S 9.12]{Tit86} for simpler proofs). The best known constant $A$ is due to Hall and Hayman \cite[Theorem 1]{HH00}:
\begin{equation}\label{Hall Hayman}
    \gamma_{n+1}-\gamma_n\leq \left(\frac{\pi}{2}+o(1)\right) \frac{1}{\log\log\log \gamma_n}, \qquad n\to \infty.
\end{equation}
More precisely, they showed that
\[
\min_\gamma |\gamma-T| \leq \frac{\pi/4}{\log\log\log T - \log\log\log\log\log T-32}
\]
for all large $T$. In addition to standard complex-analytic tools, their proof borrows concepts and ideas from hyperbolic geometry, as we shall see later in this article.

Our focus will be on unconditional results, but it is worth mentioning that under the assumption of the Riemann hypothesis (RH), i.e., if all the non-trivial zeros of $\zeta(s)$ lie on the critical line $\Re(s)=1/2$, Littlewood \cite{Lit24b} showed that the bound on the error term $S(T)$ in \eqref{Riemann-von mangoldt} can be improved from $O(\log T)$ to $O(\frac{\log T}{\log\log T})$ (see \cite{GolGon} and \cite{CCM2013} for sharper explicit constants in the last estimate). It follows immediately that 
\[
\gamma_{n+1}-\gamma_n\ll \frac{1}{\log\log \gamma_n}.
\]
Although this is already much better than what we can prove unconditionally, its order of magnitude is still far from what is expected to be true. Indeed, Farmer, Gonek and Hughes \cite[Conjecture B]{FGH07} conjectured based on random matrix theory models that the maximum size of $S(T)$ should be of order $\sqrt{\log T \log\log T}$. This also applies to $-S(T)$, so that the largest possible gap between $\gamma_{n}$ and $\gamma_{n+1}$ should be roughly $\sqrt{\frac{\log\log \gamma_n}{\log \gamma_n}}$.

The Riemann zeta-function can be extended in various ways. For instance, consider the Dirichlet $L$-function $L(s,\chi)$ associated to a primitive Dirichlet character $\chi\bmod q$. Define for $\Re(s)>1$
\[
L(s,\chi)=\sum_{n=1}^\infty \frac{\chi(n)}{n^s},
\]
which admits an analytic continuation to all of $\bb{C}$. In fact the above series converges whenever $\Re(s)>0$. Order the ordinates of its non-trivial zeros by $\ldots \leq \gamma_{\chi, -2}\leq \gamma_{\chi, -1} <0\leq \gamma_{\chi, 0}\leq \gamma_{\chi,1} \leq \ldots$, and let $N(T,\chi)$ be the number of such zeros with $-T\leq \gamma_\chi \leq T$. Similar to \eqref{Riemann-von mangoldt}, we have (see, e.g., \cite[Ch.16]{Dav})
\begin{equation}\label{Rie-Von L(s,chi)}
    N(T,\chi)=\frac{T}{\pi}\log \frac{qT}{2\pi e}+O(\log q(T+3))
\end{equation}
where the implied constant is uniform in $q$ and $T>0$. Therefore, if we fix a large constant $T$, then increasing the conductor $q$ also increases the density of zeros in the rectangle $\{s=\sigma+it: 0<\sigma<1, -T\leq t\leq T\}$. However, when $T$ is too small, we can no longer tell from \eqref{Rie-Von L(s,chi)} whether this is still the case. (Analogously, if $H>0$ is sufficiently small, then \eqref{Riemann-von mangoldt} fails to guarantee that $\zeta(s)$ always has a zero with $T\leq \gamma\leq T+H$.) Nevertheless, this is confirmed by the following result of Siegel \cite[Theorem IV]{Sie45} in 1945, an analogue of \eqref{Littlewood bound}:
\begin{equation}\label{Siegel}
    \gamma_{\chi, n+1}-\gamma_{\chi, n}\leq \frac{\pi+o(1)}{\log\log\log q}, \qquad q\to \infty
\end{equation}
for any fixed $n$. (Note that this estimate only captures the decay in the $q$-aspect and not the $t$-aspect.) In particular, it follows that
\begin{equation}\label{Siegel lowest zero}
    \min_{\gamma_\chi} |\gamma_\chi|\leq \left(\frac{\pi}{2}+o(1)\right) \frac{1}{\log\log\log q},
\end{equation}
so that the first zero of $L(s,\chi)$ gets arbitrarily close to the real axis as $q\to \infty$. Siegel's method is ingeniously simple. In comparison to Littlewood's proof of \eqref{Littlewood bound}, he used nothing more than the maximum modulus principle from complex analysis, yet the constant obtained is much sharper. However, Siegel's result appears to have been less well-known. In fact it directly answers a question raised in a recent paper by Bennett, Martin, O’Bryant and Rechnitzer \cite[Conjecture 1.3]{BMOR21}. 

To parallel the discussion for zeta, we briefly remark that conditionally on the Generalized Riemann hypothesis (GRH) for $L(s,\chi)$, one can improve the big-$O$ error term in \eqref{Rie-Von L(s,chi)} to $O(\frac{\log q(T+3)}{\log\log q(T+3)})$ by modifying Littlewood's proof \cite{Lit24b}. This yields
\[
\gamma_{\chi, n+1}-\gamma_{\chi, n}\ll \frac{1}{\log\log q}.
\]

\subsection{Motivation and setting}
Observe that the constant $\pi$ in Siegel's estimate \eqref{Siegel} for $L(s,\chi)$ is off by a factor of 2 compared to Hall and Hayman's estimate \eqref{Hall Hayman} for $\zeta(s)$. Due to the analogy between $\zeta(s)$ in the $t$-aspect and $L(s,\chi)$ in the $q$-aspect, it is natural to speculate that Hall and Hayman's method can be applied to sharpen \eqref{Siegel}. However, despite the extensive literature on zero spacings (and low-lying zeros) of Dirichlet $L$-functions, it appears that \eqref{Siegel} has never been updated, and nor were analogous unconditional estimates discussed in the contexts of more general $L$-functions. The conditional analogue, on the other hand, has been studied in a handful of references including \cite{Omar14, CCM2015, CarFin15, Zhao25}. GRH offers a significant advantage by enabling the application of the explicit formulas, which relate a sum over zeros of an $L$-function to a sum over primes. Hence, the goal of this note is to provide unconditional extensions of \eqref{Hall Hayman} and \eqref{Siegel} to a large class of $L$-functions that might be of interest and to compare the applicability and effectiveness of the two methods.

Throughout let $s=\sigma+it$ denote a generic complex number. We work with an $L$-function $L(s)$ with the following properties:
\begin{enumerate}[label=(\alph*)]
    \item It admits an absolutely convergent Euler product of the form
    \[
    L(s)=\prod_p \prod_{j=1}^m\left(1-\frac{\alpha_{j}(p)}{p^s}\right)^{-1},\qquad \sigma>1
    \]
    where 
    \begin{equation}\label{bound local parameter}
        |\alpha_{j}(p)|\leq p^\vartheta
    \end{equation}
    for some constant $0\leq \vartheta\leq 1$. Here $m$ is called the \textit{degree} of $L(s)$.

    \item For some positive integer $N_L$ called the \textit{conductor}, and complex numbers $\{\mu_j\}_{1\leq j\leq m}=\{\overline{\mu_j}\}_{1\leq j\leq m}$ with $\Re(\mu_j)>-1$, the completed $L$-function 
    \[
    \xi_L(s)= L(s) {N_L}^{s/2}\prod_{j=1}^{m} \Gamma_\bb{R}(s+\mu_j), \quad \Gamma_\bb{R}(s)=\pi^{-s/2}\Gamma(s/2)
    \]
    extends to a meromorphic function on $\bb{C}$ of order 1. The only possible poles of $\xi_L(s)$ are located at $s=0$ and $s=1$, each of order $0\leq r_L\leq m$. We also assume that $L(s)$ has no pole at $s=0$, so that its only pole is at $s=1$ of order $r_L$. Moreover, $\xi_L(s)$ satisfies the functional equation    
    \begin{equation}\label{functional equation}
        \xi_L(s)=\kappa \ov{\xi_L(1-\ov{s})}, \quad |\kappa|=1.
    \end{equation}
\end{enumerate}

\begin{remark}
    The main examples are $L$-functions attached to automorphic representations of $\mr{GL}_m$ over number fields. In this setting, according to the work of Luo, Rudnick and Sarnak \cite{LRS1995}, we have $\vartheta\leq \frac{1}{2}-\frac{1}{m^2+1}$ and $\Re(\mu_j)\geq -\frac{1}{2}+\frac{1}{m^2+1}$. The generalized Ramanujan\textendash Petersson conjecture asserts that $\vartheta=0$.
\end{remark}

It follows from the absolute convergence of the Euler product that $L(s)$ has no zero in the half-plane $\sigma>1$. Non-vanishing on the 1-line is much deeper and has been proven in the most interesting cases, but not in full generality. Also note that $L(s)$ must vanish at $\{-2k-\mu_j\}_{k\in \mathbb{N}_{\geq 0}}$ for each $\mu_j\neq 0$ and $\{-2k\}_{k\in \mathbb{N}_{\geq 1}}$ if $\mu_j=0$ for some $j$, but $L(0)$ can either be zero or nonzero depending on $r_L$ and the cardinality of $\{1\leq j\leq m: \mu_j=0\}$. We call these zeros arising from the gamma factors \textit{trivial zeros} of $L(s)$, and all other zeros are \textit{non-trivial zeros}. Our assumption $\Re(\mu_j)>-1$ allows the existence of trivial zeros throughout $\sigma<1$. The non-trivial zeros are located in the critical strip $0\leq \sigma\leq 1$. 

We shall measure the spacings between zero ordinates near $T$ in terms of the \textit{analytic conductor}
\begin{equation}\label{analytic conductor}
    C_L(T)=N_L \prod_{j=1}^m(|iT+\mu_j|+3).
\end{equation}

\subsection{Results}
Our first result is an extension of \eqref{Hall Hayman}:

\begin{theorem}\label{thm HH}
    When $\frac{\log\log C_L(T)}{\log(m+2)}$ is sufficiently large, $L(s)$ possesses a non-trivial zero $\rho=\beta+i\gamma$ with 
    \[
    |\gamma-T|\leq \frac{\pi}{4}\dfrac{1+2\vartheta}{\log\left(\frac{\log\log C_L(T)}{\log(m+2)}\right)-\log\log\log\left(\frac{\log\log C_L(T)}{\log(m+2)}\right)-(4+\log 2+o(1))}.
    \]
\end{theorem}

\begin{remark}
    In addition to \eqref{Hall Hayman}, \cite{HH00} provides estimates on the number of zeta zeros in various rectangular and non-rectangular regions. These results can be extended to our setting as well, but we shall not pursue them here. The reader should consult \cite{HH00} if they are interested in questions such as ``(when $q$ is large) what is the least number of non-trivial zeros of $L(s,\chi)$ with $|\gamma_\chi|<\frac{1}{100}$?" 
\end{remark}

We give two concrete applications of Theorem~\ref{thm HH}.
\begin{example}[$m=1$]
    For $L(s,\chi)$ where $\chi\bmod q$ is a primitive Dirichlet character, the completed $L$-function is
    \[
    \xi(s,\chi)=\left(\frac{q}{\pi}\right)^{s/2}\Gamma\left(\frac{s+\mathfrak{a}}{2}\right)L(s,\chi)
    \]
    where $\mathfrak{a}$ is $0$ or $1$ according as $\chi$ is even or odd.
    Then
    \[
    \min_{\gamma_\chi} |\gamma_\chi|\leq \left(\frac{\pi}{4}+o(1)\right)\frac{1}{\log\log\log q}
    \]
    as $q\to \infty$, improving the leading constant in \eqref{Siegel lowest zero} by a factor of 2. 
\end{example}

\begin{example}[$m=2$]
    Let $f\in S_k(\Gamma_0(N))$ be a normalized newform of weight $k$ and level $N$, and let $L(s,f)$ be the associated $L$-function. The completed $L$-function is
    \[
    \xi(s,f)=L(s,f)\bigg(\frac{N}{4\pi^2}\bigg)^{s/2} \Gamma\left(s+\frac{k-1}{2}\right).
    \]
    Then, as $k^2N\to \infty$,
    \[
    \min_{\gamma_f} |\gamma_f|\leq \left(\frac{\pi}{4}+o(1)\right)\frac{1}{\log\log\log k^2N}.
    \]
\end{example}

Unfortunately, Theorem~\ref{thm HH} suffers from the limitation that it is only applicable when $\frac{\log\log C_L(T)}{\log(m+2)}$ is large, and in order to obtain arbitrarily small gaps between zeros we need this quantity to tend to infinity. Ideally we want an estimate that applies to all $L$-functions in our general class given an arbitrary $T$, or at least when $|T|$ is sufficiently large. For simplicity we can assume that all the $\mu_j$ are real (or that their imaginary parts are bounded). Note however that no matter how large we take $|T|$, there possibly exists an $L$-function whose conductor $N_L$ is not large enough relative to its degree $m$ such that the quantity $\frac{\log\log C_L(T)}{\log(m+2)}$ is small. This means that Theorem~\ref{thm HH} is not suitable for handling families of $L$-functions with rapid degree growth (e.g., Dedekind zeta-functions; see below). To this end, we provide the following extension of Siegel's result \eqref{Siegel}, which has a less sharp leading constant but turns out to address the aforementioned issue.

\begin{theorem}\label{thm Siegel}
    When $C_L(T)^{1/m}$ is sufficiently large, $L(s)$ possesses a non-trivial zero $\rho=\beta+i\gamma$ with 
    \[
    |\gamma-T|\leq \frac{\pi}{2} \frac{1+2\vartheta}{\log\log\log C_L(T)^{1/m}}+O\left(\frac{1}{(\log\log\log C_L(T)^{1/m})^2}\right)
    \]
    where the implied constant is absolute. 
\end{theorem}

Compared to Theorem~\ref{thm HH}, Theorem~\ref{thm Siegel} is more applicable and is sharper when $(\log m)^2\gtrapprox \log\log C_L(T)$, or $(\log m)^2\gtrapprox \log\log N_L$ if $|iT+\mu_j|$ are small. 

\begin{example}
    For a number field $\bb{K}$ with $r_1$ real embeddings and $2r_2$ complex embeddings, let $\Delta_{\bb{K}}$ be the discriminant, $n_{\bb{K}}=[\bb{K}:\bb{Q}]$ be the degree of extension, and $\zeta_{\bb{K}}(s)$ be the associated Dedekind zeta-function. Then the completed $L$-function is
    \[
    \xi_{\bb{K}}(s)=\left(\frac{\Delta_{\bb{K}}}{4^{r_2}\pi^{n_{\bb{K}}}}\right)^{s/2}\Gamma\left(\frac{s}{2}\right)^{r_1}\Gamma(s)^{r_2}\zeta_{\bb{K}}(s).
    \]
    When the root discriminant $\mr{rd_{\bb{K}}}=\Delta_{\bb{K}}^{1/n_{\bb{K}}}$ is sufficiently large, 
    \[
    \min_{\gamma_{\bb{K}}}|\gamma_{\bb{K}}|\leq \left(\frac{\pi}{2}+o(1)\right)\frac{1}{\log\log\log \mr{rd_{\bb{K}}}}.
    \]
    Note that it is possible for $\mr{rd_{\bb{K}}}\to \infty$ while
    $\frac{\log\log \Delta_{\bb{K}}}{\log n_{\bb{K}}}\to 1$. 
\end{example}

Theorem~\ref{thm Siegel} implies that gaps between zeros of $L$-functions are uniformly bounded, which may alternatively be derived from the zero-counting formula for $L(s)$ in the spirit of \eqref{Riemann-von mangoldt} and \eqref{Rie-Von L(s,chi)}.

\begin{corollary}\label{corollary}
    There exists a universal constant $N_0>0$ such that each rectangle of the form $0\leq \sigma\leq 1$, $T-N_0\leq t\leq T+N_0$ contains a non-trivial zero of $L(s)$.
\end{corollary}

Here we do not attempt to make the constant $N_0$ explicit. Under GRH and several other restrictions on $L(s)$, such explicit estimates have been obtained in \cite{Miller02, Boberetal15, KRZ17} via applications of the explicit formulas to test functions with certain positivity properties.

\begin{proof}[Proof of Corollary~\ref{corollary}]
    Suppose that $[T-N_0, T+N_0]$ contains no zero ordinate with $N_0>10$, say. Set $T_1=T-5$ and $T_2=T+5$. Since both intervals centered at $T_1$ and $T_2$ with radius 5 are zero-free, Theorem ~\ref{thm Siegel} implies that $C_L(T_1)^{1/m}$ and $C_L(T_2)^{1/m}$ are both bounded by some universal constant $N_1$. Hence, 
    \begin{align*}
        2\log N_1\geq& \log C_L(T_1)^{1/m}+\log C_L(T_2)^{1/m}\\
        \geq& \min_{1\leq j\leq m }[\log(3+|iT_1+\mu_j|)+\log(3+|iT_2+\mu_j|)].
    \end{align*}
    Let $j'$ be the index at which the minimum is attained. It follows that
    \[
    N_0=T_2-T_1\leq |T_1+\Im(\mu_{j'})|+|T_2+\Im(\mu_{j'})|\leq N_1^2.
    \]
\end{proof}

To prove the two main theorems, we largely follow the original proofs of \cite[Theorem 1]{HH00} and \cite[Theorem IV]{Sie45}. While the former proof is much more involved than the latter, the essential strategy common to both arguments is to construct and estimate certain auxiliary functions involving $L(s)$ that are holomorphic in certain rectangular regions. In our general setting, more care is needed mostly due to the randomness of the parameters $\mu_j$, which affect a number of things including locations of the trivial zeros, applicability of Stirling's formula when estimating the gamma factors, etc.

\section{Preparations}

\subsection{Estimating the size of $L(s)$}
Following the ideas in \cite{PalZha25}, we first prove an explicit bound on $L(s)$ that will be used in the proofs of both theorems and might also be of independent interest. For this we need a modified version of the Phragm\'{e}n\textendash Lindel\"{o}f principle due to Rademacher:

\begin{lemma}[{\cite[Theorem 2]{Rad59}}]\label{lemma Rademacher}
    Let $g(s)$ be an analytic function of finite order in the strip $S(a,b)=\{s: a\leq \sigma\leq b\}$ such that
    \[
    \begin{cases}
        |g(a+it)|\leq A\prod_{j=1}^m |Q_j+a+it|^\alpha \\
        |g(b+it)|\leq B\prod_{j=1}^m |Q_j+b+it|^\beta
    \end{cases}
    \]
    with $\Re(Q_j)+a>0$ and $\alpha\geq \beta$. Then for any $s\in S(a,b)$,
    \[
    |g(s)|\leq \Bigg(A\prod_{j=1}^m|Q_j+s|^\alpha\Bigg)^{\frac{b-\sigma}{b-a}} \Bigg(B\prod_{j=1}^m|Q_j+s|^\beta\Bigg)^{\frac{\sigma-a}{b-a}}.
    \]
\end{lemma}

This allows us to deduce the following result:

\begin{lemma}\label{lemma: |L(s)| bound}
    Let $0\leq d\leq \mathrm{ord}_{s=0}L(s)$ where $\mathrm{ord}_{s=0}L(s)$ denotes the order of vanishing of $L(s)$ at $s=0$,
    \[
    J_1\subset \{1\leq j\leq m: -1<\Re(\mu_j)\leq 3/2, \:\mu_j\neq 0\},
    \]
    \[
    J_2\subset J_1\cap \{1\leq j\leq m: -1<\Re(\mu_j)\leq -1/2\}.
    \]  
    Then the function 
    \[
    g(s)=\frac{L(s)(s-1)^{r_L}}{s^{d}\prod_{j\in J_1}(s+\mu_j)\prod_{j\in J_2}(s+2+\mu_j)}
    \]
    is bounded on $\sigma\in [-3/2,5/2]$ by
    \[
    |g(s)|<e^m\Bigg|N_L\prod_{j=1}^m \frac{s+\frac{5+\sqrt{13}}{2}+\mu_j}{2\pi}\Bigg|^{\frac{5/2-\sigma}{2}}|s-4|^{r_L}.
    \]
\end{lemma}
Note that $g(s)$ is analytic and the denominator allows for flexibility if one wants to remove the trivial zeros of $L(s)$ in the strip $\sigma\in [-3/2,1)$. Our bound is far from being sharp but suffices for our applications. For example, when $\sigma=1/2$ the exponent $\frac{5/2-\sigma}{2}=1$ should be replaced by $1/4$ according to the convexity bound. See \cite{PalZha25} for a slightly more refined application of Lemma~\ref{lemma Rademacher} toward this end.

\begin{proof}
    First observe that
    \[
    \log L(s)=\sum_{n=1}^\infty \frac{\Lambda_L(n)}{n^s \log n}, \qquad \sigma>1
    \]
    where the Dirichlet series converges absolutely and its coefficients satisfy
    \begin{equation}\label{generalized von mangoldt}
        |\Lambda_L(n)|\leq m\Lambda(n)n^\vartheta.
    \end{equation}
    Here $\vartheta$ was introduced in \eqref{bound local parameter} and $\Lambda(n)$ denotes the von Mangoldt function, which takes the value $\log p$ when $n=p^k$ where $k\geq 1$ and 0 otherwise. Since
    \begin{equation*}
        \log \zeta(s)=\sum_{n=1}^\infty \frac{\Lambda(n)}{n^s \log n}, \qquad \sigma>1,
    \end{equation*} 
    we have 
    \[
    \left|\log L\left(\frac{5}{2}+it\right)\right|\leq \sum_{n=1}^\infty \frac{m\Lambda(n)n^\vartheta}{n^{5/2}\log n}=m\log \zeta\left(5/2-\vartheta\right)\leq m\log \zeta(3/2)<m,
    \]
    and hence $|L(s)|<e^m$ on $\sigma=5/2$. First suppose that $L(s)$ is entire, i.e., $r_L=0$. Trivially,
    \[
    \left|g\left(\frac{5}{2}+it\right)\right|\leq \frac{e^m}{(5/2)^d \prod_{j\in J_1}|\mu_j+5/2+it| \prod_{j\in J_2}|\mu_j+9/2+it|}<e^m.
    \] 
    From the functional equation \eqref{functional equation}, the recurrence relation 
    \begin{equation}\label{gamma relation}
        s\Gamma(s)=\Gamma(s+1),
    \end{equation}
    and our assumption $\Re(\mu_j)>-1$, we also deduce that
    \begin{align*}
        &\left|g\left(-\frac{3}{2}+it\right)\right|\\
        &=
        \left|\frac{L(5/2+it)}{(-3/2+it)^d}\right| \frac{N_L^2}{\pi^{2m}} \prod_{j\in J_1\setminus J_2} \bigg|\frac{\Gamma(\frac{\mu_j+5/2+it}{2})}{2\Gamma(\frac{\mu_j+1/2+it}{2})}\bigg|       
        \prod_{j\in J_2} \bigg|\frac{\Gamma(\frac{\mu_j+5/2+it}{2})}{4\Gamma(\frac{\mu_j+5/2+it}{2})}\bigg| 
        \prod_{j\not \in J_1}\bigg|\frac{\Gamma(\frac{\mu_j+5/2+it}{2})}{\Gamma(\frac{\mu_j-3/2+it}{2})}\bigg|\\
        &= \left|\frac{L(5/2+it)}{(-3/2+it)^d}\right|\frac{N_L^2}{\pi^{2m}} \prod_{j\in J_1\setminus J_2}\left|\frac{\mu_j+1/2+it}{4}\right| \prod_{j\in J_2}\frac{1}{4}\prod_{j\not \in J_1}\left|\frac{\left(\mu_j-3/2+it\right)\left(\mu_j+1/2+it\right)}{4}\right|\\
        &<e^m\frac{N_L^2}{(2\pi)^{2m}} \prod_{j=1}^m \left|Q_j-\frac{3}{2}+it\right|^2
    \end{align*}
    where $Q_j=\mu_j+\frac{5+\sqrt{13}}{2}$. We obtain the desired estimate by applying Lemma~\ref{lemma Rademacher} with $a=-3/2$, $b=5/2$, $A=e^m N_L^2(2\pi)^{-2m}$, $B=e^m$, $\alpha=2$ and $\beta=0$. In general, since $|s-1|\leq |s-4|$ on the half-plane $\sigma\leq 5/2$, the lemma follows after we replace $L(s)$ by $L(s)(\frac{s-1}{s-4})^{r_L}$ in the above argument.
\end{proof}

We also prove the following lemma.

\begin{lemma}[c.f. {{\cite[Lemma 6]{HH00}}}]\label{lemma: log L(s_0)/L(s_1)}
    Let $s_0,s_1\in \bb{C}$ such that $\Re(s_1)=1+\vartheta+a$ and $|s_0-s_1|\leq (1-e^{-1/m})a$ for some $a>0$. Then
    \[
    \left|\log \frac{L(s_0)}{L(s_1)}\right|<1.
    \]
\end{lemma}

\begin{proof}
    We have
    \[
    \frac{\md}{\md s}\log L(s)=\frac{L'}{L}(s)=-\sum_{n=1}^\infty \frac{\Lambda_L(n)}{n^s \log n}, \qquad \sigma>1
    \]
    where $\Lambda_L(n)$ satisfies \eqref{generalized von mangoldt}.
    Since $\Re(s_0)$ and $\Re(s_1)$ are both $>1$, 
    \begin{align*}
        \frac{1}{|s_0-s_1|}\left|\log \frac{L(s_0)}{L(s_1)}\right| =&\left|\int_0^1 \frac{L'}{L}(s_1+t(s_0-s_1))\md t\right|\\
        \leq& \int_0^1 \sum_{n=1}^\infty \frac{|\Lambda_L(n)|}{n^{\Re(s_1)+t\Re(s_0-s_1)}}\md t\\
        \leq& \int_0^1 \sum_{n=1}^\infty \frac{m\Lambda(n)n^{\vartheta}}{n^{\Re(s_1)-t|s_0-s_1|}}\md t\\
        =& \int_0^1 -m\frac{\zeta'}{\zeta}(1+a-t|s_0-s_1|)\md t,
    \end{align*}
    where we invoked \eqref{generalized von mangoldt} to transition from $L'/L$ to 
    \[
    \frac{\zeta'}{\zeta}(s)=-\sum_{n=1}^\infty \frac{\Lambda(n)}{n^s}, \qquad \sigma>1.
    \]
    Using the inequality 
    \[
    -\frac{\zeta'}{\zeta}(\sigma)<\frac{1}{\sigma-1}, \qquad \sigma>1
    \]
    (a proof is included in the Appendix), we conclude that
    \begin{align*}
        \left|\log \frac{L(s_0)}{L(s_1)}\right|<& |s_0-s_1| \int_0^1 \frac{m}{a-t|s_0-s_1|}\md t
        \leq m\log\left(\frac{a}{a-|s_0-s_1|}\right) \leq 1.
    \end{align*}
\end{proof}

\subsection{Hyperbolic distance and preparation for Theorem~\ref{thm HH}}

The setup for Theorem~\ref{thm HH} comes from the Poincar\'{e} disk model in two-dimensional hyperbolic geometry. Let $G\subset \bb{C}$ be a simply connected domain and let $D=\{z\in \bb{C}: |z|<1\}$ be the open unit disk. For any two points $w_1,w_2\in G$, the Riemann mapping theorem implies that there exists a unique biholomorphic map from $G$ to $D$ that maps $w_1$ to 0 and $w_2$ to some $r\in (0,1)$. We define the hyperbolic distance between $w_1$ and $w_2$ (relative to $G$) to be
\begin{equation}\label{def hyp dis}
    d(w_1,w_2; G)=\frac{1}{2}\log \frac{1+r}{1-r},
\end{equation}
or equivalently, $d(w_1,w_2; G)=\tanh^{-1}(r)$ where $\tanh(x)=\frac{e^x-e^{-x}}{e^x+e^{-x}}$ is the hyperbolic tangent function. A useful fact is that expanding the domain decreases the hyperbolic distance between two points, that is, if $G\subset G'$, then $d(w_1,w_2; G)\geq d(w_1,w_2; G')$. This can be seen by applying the Schwarz\textendash Pick lemma to the holomorphic function formed by the composition $D\to G\xhookrightarrow{} G'\to D$.

For convenience we record several lemmas from \cite{HH00}. It is worth remarking that their proofs, which we omit here, are in fact quite short and elegant. The first is a Borel\textendash Carath\'{e}odory type result, proved using Hadamard's three circles theorem.

\begin{lemma}[{\cite[Theorem 4]{HH00}}]\label{lemma: caratheodory}
    Let $f:D\to \bb{C}$ be analytic with $f(0)=0$. Put 
    \[
    K(f)=\sup\{\Re(f(z)): z\in D\}
    \]
    and
    \[
    M(r,f)=\max\{|f(z)|: |z|\leq r\}, \qquad r<1.
    \]
    If
    \[
    M(r_1,f)=K(f)r_1^c
    \]
    for some $r_1\in (0,1)$ and $c>0$, then
    \[
    M(r_2,f)<2K(f)\frac{r_2^c}{1-r_2^c}
    \]
    for all $r_2\in (r_1,1)$.
\end{lemma}

The next two lemmas are related to properties of hyperbolic distance in rectangular domains.
\begin{lemma}[{\cite[Lemma 2]{HH00}}]\label{lemma: hyper dist horiz strip}
    Let $s_0,s_1\in S=\{s: |t-T|<a\pi/4\}$ with $\Im(s_1)=T$. Then
    \[
    d(s_0,s_1;S)\geq \frac{|s_0-s_1|}{a}.
    \]
\end{lemma}

\begin{lemma}[{\cite[Lemma 5]{HH00}}]\label{lemma: hyper dist rectangle}
    Let $\sigma_0>1$ and $R=\{s: 0<\sigma<\sigma_0, |t|\leq \pi/4\}$. Then for all $(\log 2)/2<x\leq \sigma_0/2$,
    \[
    d(x,\sigma_0-x;R)<\sigma_0.
    \]
\end{lemma}

In the next two sections we prove Theorem~\ref{thm HH} and \ref{thm Siegel}.

\section{Proof of Theorem~\ref{thm HH}}

Throughout we assume (and may not explicitly mention at each instance) that $\frac{\log\log C_L(T)}{\log(m+2)}$ is large. 
Suppose that $L(s)$ has no non-trivial zero with $|\gamma-T|\leq \frac{\pi}{4}a(T)$ where 
\[
\frac{1}{\log\log\log C_L(T)}\leq a(T)\leq \frac{1}{10}.
\]
Put
\begin{equation}\label{def s_i}
    s_1=1+\vartheta+a(T)+iT, \quad s_3=\frac{1}{2}-a(T)+iT, \quad s_4=\frac{1}{2}+a(T)+iT.
\end{equation}
For now let $s_2$ be either $s_3$ or $s_4$; the exact choice will only matter later on.

Consider the rectangle
\begin{equation}\label{def: R rect}
    R=\left\{s: \frac{1}{2}-\frac{2+\log 2}{2}a(T)<\sigma<1+\vartheta+\frac{2+\log 2}{2}a(T), \: |t-T|<\frac{\pi}{4}a(T)\right\},
\end{equation}
which contains $s_1,s_2$ and no non-trivial zero. Lemma~\ref{lemma: hyper dist rectangle} gives (after rescaling) 
\[
d(s_1,s_2;R)<\frac{1}{a(T)}
\left(1+\vartheta+\frac{2+\log 2}{2}a(T)-\frac{1}{2}+\frac{2+\log 2}{2}a(T) \right)=\frac{1+2\vartheta}{2a(T)}+(2+\log 2).
\]
Take any Riemann mapping $\phi:D\to R$ such that $\phi(0)=s_1$ and put $r_2=|\phi^{-1}(s_2)|$. Then, by the definition of hyperbolic distance \eqref{def hyp dis} we have
\[
\frac{1}{2}\log \frac{1+r_2}{1-r_2}=d(s_1,s_2;R),
\]
so that
\begin{equation}\label{bound (1+r_2)/(1-r_2)}
    \frac{1+r_2}{1-r_2}<e^{(1+2\vartheta)/a(T)+(4+2\log 2)}.
\end{equation}
    
Next we define an auxiliary function that is analytic in $D$. Select the subset of indices 
\[
J=\left\{1\leq j\leq m: -1<\Re(\mu_j)<-\frac{1}{2}+\frac{2+\log 2}{2}a(T), \: |T+\Im(\mu_j)|<\frac{\pi}{4}a(T)\right\}
\]
and put
\begin{equation}\label{def r_L(T)}
    r_L(T)=
    \begin{cases}
        r_L &\text{if $|T|\leq 1$},\\
        0 &\text{if $|T|>1$}.
    \end{cases}
\end{equation}
Then define for $z\in D$
\begin{equation}\label{def f}
    f(z)=\log \frac{g(\phi(z))}{g(s_1)} \quad \text{where}\:\: g(s)=\frac{L(s)(s-1)^{r_L(T)}}{\prod_{j\in J}(s+\mu_j)}.
\end{equation}
Note that $g$ is analytic in $R$ because the trivial zeros of $L(s)$ and the possible pole at $s=1$ are all canceled out provided that they lie in $R$. Consequently $f$ is analytic in $D$ with $f(0)=0$. We establish two claims regarding the quantities $M(r,f)$ and $K(f)$ defined in Lemma~\ref{lemma: caratheodory}. 
\begin{lemma}\label{claim: M(r_1,f)}
    \[
    M(r_1,f)<4 \quad \text{where}\:\: r_1=\tanh(1-e^{-\frac{1}{m+1}}).
    \]
\end{lemma}
    
\begin{proof}
    Write $f(z)=f_1(z)+f_2(z)+f_3(z)$ where
    \[
    f_1(z)=\log \frac{L(\phi(z))}{L(s_1)},\quad f_2(z)=\log \frac{\prod_{j\in J}(s_1+\mu_j)}{\prod_{j\in J}(\phi(z)+\mu_j)}, \quad f_3(z)=r_L(T) \log \left(\frac{\phi(z)-1}{s_1-1}\right).
    \]
    We claim that 
    \begin{equation}\label{bound |s_0-s_1|}
        |\phi(z)-s_1|\leq (1-e^{-\frac{1}{m+1}})a(T)
    \end{equation}
    for all $|z|\leq r_1$ (the proof will be given later). Since $|s_1+\mu_j|>\vartheta+a(T)\geq a(T)$, it follows from \eqref{bound |s_0-s_1|} that \[
    |\phi(z)+\mu_j|\geq |s_1+\mu_j|-|s_1-\phi(z)|\geq e^{-\frac{1}{m+1}}a(T)>0,
    \]
    and similarly $|\phi(z)-1|>0$,
    so that each $f_i$ is analytic in $|z|\leq r_1$. Lemma~\ref{lemma: log L(s_0)/L(s_1)} and \eqref{bound |s_0-s_1|} imply that $M(r_1,f_1)\leq 1$.
    Next we show $M(r_1,f_2)<2$. Writing $\log(1+w)=\int_0^1 \frac{w}{1+wx}\md x$, it is not hard to verify that 
    \begin{equation}\label{|log(1+z)|}
        |\log(1+w)|\leq \frac{|w|}{1-|w|}, \qquad w\in \mathbb{C}, \:|w|\in [0,1).
    \end{equation}
    Therefore, since
    \[
    \left|\frac{s_1-\phi(z)}{\phi(z)+\mu_j}\right|\leq \frac{1-e^{-\frac{1}{m+1}}}{1-(1-e^{-\frac{1}{m+1}})}=e^{\frac{1}{m+1}}-1<1,
    \] 
    we obtain
    \[
    |f_2(z)|\leq \sum_{j\in J}\left|\log\left(1+\frac{s_1-\phi(z)}{\phi(z)+\mu_j}\right)\right|\leq m\frac{(e^{\frac{1}{m+1}}-1)}{1-(e^{\frac{1}{m+1}}-1)}<2
    \] 
    for all $m\geq 1$, as desired. Similarly, 
    \[
    |f_3(z)|\leq m\left|\log\left(1+\frac{\phi(z)-s_1}{s_1-1}\right)\right|\leq m(e^{\frac{1}{m+1}}-1)<1,
    \]
    and thus $M(r_1,f_3)<1$. Adding the contributions from each $f_i$ gives the claim.
    
    It remains to prove \eqref{bound |s_0-s_1|}. Since $R\subset S:=\{s: |t-T|<\frac{\pi}{4}a(T)\}$ and hyperbolic distance shrinks in expanding domains, we have
    \[
    d(\phi(z),s_1; R)\geq d(\phi(z),s_1; S)\geq \frac{|\phi(z)-s_1|}{a(T)}
    \]
    where we applied Lemma~\ref{lemma: hyper dist horiz strip} for the second inequality. Hence, if \eqref{bound |s_0-s_1|} fails for some $z$ with $|z|\leq r_1$, then
    \[
    d(\phi(z),s_1;R)>1-e^{-\frac{1}{m+1}}=\frac{1}{2}\log\frac{1+r_1}{1-r_1}\geq \frac{1}{2}\log\frac{1+|z|}{1-|z|}=d(\phi(z),s_1;R),
    \]
    where a contradiction arises. 
\end{proof}
    
\begin{lemma}\label{claim: K(f)}
    \[
    K(f)<\frac{3}{2}\log C_L(T).
    \]
\end{lemma}
\begin{proof}
    By definition, 
    \begin{equation}\label{expression K(f)}
        K(f)=\sup_{s\in R} \log \left|\frac{L(s)(s-1)^{r_L(T)}}{\prod_{j\in J} (s+\mu_j)}\right|+\sum_{j\in J}\log|s_1+\mu_j|-\log |L(s_1)|-r_L(T)\log|s_1-1|.
    \end{equation}
    Fix $s\in R$. According to Lemma~\ref{lemma: |L(s)| bound}, the first quantity does not exceed
    \[
    m+\frac{5/2-\sigma}{2} \log \Bigg|N_L\prod_{j=1}^m \frac{s+\frac{5+\sqrt{13}}{2}+\mu_j}{2\pi}\Bigg|+
    \begin{cases}
        r_L\log|s-4| & \text{if $|T|\leq 1$},\\
        r_L\log \left|\dfrac{s-4}{s-1}\right| & \text{if $|T|>1$}.
    \end{cases}
    \]
    This estimate (especially the coefficient $\frac{5/2-\sigma}{2}$) is crude and can be improved with more effort, but it suffices for our purposes. The second term above is at most
    \[
    \left(1+\frac{2+\log 2}{4}a(T)\right)\log C_L(T)<1.1\log C_L(T)
    \]
    since $a(T)\leq 1/10$, and the third term is $O(m)$ in either case.
    
    For the last three terms in \eqref{expression K(f)}, we estimate $\sum_{j\in J}\log|s_1+\mu_j|=O(m)$,
    \[
    -r_L(T)\log|s_1-1|\leq m\log \frac{1}{a(T)}\leq m\log\log\log\log C_L(T),
    \]
    and
    \begin{align}\label{log L(s_1) upper bound}
        |\log L(s_1)|\leq& \sum_{n=1}^\infty \frac{m\Lambda(n)n^\vartheta}{n^{1+\vartheta+a(T)}\log n}=m\log \zeta(1+a(T))
        <m\log \frac{1+a(T)}{a(T)}\notag \\
        <& 2m\log\log\log\log C_L(T),
    \end{align}
    where we used \eqref{generalized von mangoldt} as well as the inequality $\zeta(\sigma)< \frac{\sigma}{\sigma-1}$, $\sigma>1$ (see, e.g., \cite[Corollary 1.14]{MV}).
    Putting these estimates together, we obtain 
    \[
    K(f)<1.1 \log C_L(T)+O(m\log\log\log\log C_L(T)).
    \]
    Since $\frac{\log\log C_L(T)}{\log(m+2)}$ is assumed to be large, the error term is $o(\log C_L(T))$. This finishes the proof.
\end{proof}

Recall that we defined $s_3$ and $s_4$ in \eqref{def s_i} and it did not matter whether $s_2=s_3$ or $s_2=s_4$ in the arguments above. We now determine $s_2$. Let $g(s)$ be as in \eqref{def f}. Since $s_3=1-\overline{s_4}$, 
we find by using the functional equation \eqref{functional equation} and the gamma relation \eqref{gamma relation}
that
\begin{multline}\label{ratio L(s3) L(s4)}
    \left|\frac{g(s_3)}{g(s_4)}\right|=\frac{N_L^{a(T)}}{\pi^{m \cdot a(T)}}  \prod_{1\leq j\leq m} \left|\frac{\Gamma \big(\frac{5/2+a(T)+iT+\mu_j}{2}\big)}{\Gamma\big(\frac{5/2-a(T)+iT+\mu_j}{2}\big)}\right| \prod_{j\not \in J} \left|\frac{1/2-a(T)+iT+\mu_j}{1/2+a(T)+iT+\mu_j}\right|\\
    \times \left|\frac{-1/2-a(T)+iT}{-1/2+a(T)+iT}\right|^{r_L(T)}.
\end{multline}
By the mean-value theorem,
\begin{align*}
    \log \left|\frac{g(s_3)}{g(s_4)}\right|=a(T)&\Bigg[\log \left(\frac{N_L}{\pi^m}\right)+\sum_{j=1}^m \Re\frac{\Gamma'}{\Gamma}\left(\frac{5/2+\sigma^*_j+iT+\mu_j}{2}\right)\\
    &+2\sum_{j\not \in J}\Re\left(\frac{1}{1/2+\tau^*_j+iT+\mu_j}\right)+2r_L(T)\Re\left(\frac{1}{-1/2+\kappa^*+iT}\right)\Bigg],
\end{align*}
where $\sigma^*_j,\tau^*_j,\kappa^*\in [-a(T),a(T)]$.  
Applying Stirling's formula in the form
\begin{equation}\label{Stirling}
    \frac{\Gamma'}{\Gamma}(s)=\log(1+s)-\frac{1}{s}+O(1), \qquad \sigma>-1+\delta
\end{equation}
and using the definition of $J$, we see that
\begin{align*}
    |\log g(s_3)|+|\log g(s_4)|\geq & a(T)\left(\log \left(\frac{N_L}{\pi^m}\right)+\sum_{j=1}^m \big[\log(|iT+\mu_j|+3)+O(1)\big]+O\left(\frac{m}{a(T)}\right)\right)\\
    =& a(T)\log C_L(T)+O(m)\\
    \geq & \frac{2a(T)}{3}\log C_L(T).
\end{align*}
Now choose $s_2$ to be either $s_3$ or $s_4$ such that
\begin{equation}\label{log L(s_2) lower bound}
    |\log g(s_2)|\geq \frac{a(T)}{3}\log C_L(T).
\end{equation}

For the last step of the proof, we consider two possible cases for $K(f)$ and shall see that the desired conclusion follows in either case. 
    
\bigskip
\noindent \textbf{Case 1}. 
\begin{equation}\label{assumption case 1}
    K(f)<\frac{\log C_L(T)}{4(\log\log C_L(T))^2}.
\end{equation}
On the one hand, in view of \eqref{log L(s_1) upper bound} and \eqref{log L(s_2) lower bound}, 
\begin{equation}\label{lower bound M(r_2,f)}
    M(r_2,f)\geq |f(\phi^{-1}(s_2))|=\left|\log \frac{g(s_2)}{g(s_1)}\right|\geq \frac{a(T)}{4}\log C_L(T).
\end{equation}
On the other hand, by the Borel\textendash Carath\'{e}odory theorem, inequality \eqref{bound (1+r_2)/(1-r_2)} and assumption \eqref{assumption case 1},
\[
M(r_2,f)\leq \frac{2r_2}{1-r_2}K(f)<\frac{1+r_2}{1-r_2}K(f)<e^{(1+2\vartheta)/a(T)+(4+2\log 2)} \frac{\log C_L(T)}{4(\log\log C_L(T))^2}.
\]
Combining the upper and lower bounds yields
\[
\frac{1+2\vartheta}{a(T)}+(4+2\log 2) > 2\log\log\log C_L(T),
\]
that is,
\begin{equation}\label{bound a(T) case 1}
    a(T)<\frac{1+2\vartheta}{2\log\log\log C_L(T)-(4+2\log 2)}.
\end{equation}

\bigskip
\noindent \textbf{Case 2}.
\begin{equation}\label{assumption case 2}
K(f)\geq \frac{\log C_L(T)}{4(\log\log C_L(T))^2}.
\end{equation}
We see from Lemma~\ref{claim: M(r_1,f)} and \eqref{lower bound M(r_2,f)} that $M(r_1,f)<M(r_2,f)$, and so $r_1<r_2$. By \eqref{lower bound M(r_2,f)}, Lemma~\ref{lemma: caratheodory}, and Lemma~\ref{claim: K(f)},
\begin{equation}\label{eqn upper bound c}
    \frac{a(T)}{4}\log C_L(T)\leq M(r_2,f)<2K(f)\frac{r_2^c}{1-r_2^c}<3\log C_L(T)\frac{r_2^c}{1-r_2^c}
\end{equation}
where $c>0$ satisfies
\begin{equation}\label{eqn lower bound on c}
    M(r_1,f)=K(f)r_1^c.
\end{equation}
We first extract a lower bound on $c$ from \eqref{eqn lower bound on c}. By Lemma~\ref{claim: M(r_1,f)} and assumption \eqref{assumption case 2},
\begin{equation}\label{bound r_1^c}
    r_1^c\leq \frac{16(\log\log C_L(T))^2}{\log C_L(T)}.
\end{equation}
One can verify the inequality
\[
\tanh(x)>\frac{\log(\frac{1}{1-x})}{1+\log(\frac{1}{1-x})}, \qquad 0<x<0.936\ldots,
\]
which implies $1/r_1<m+2$ after substituting $x=1-e^{-\frac{1}{m+1}}$. We then find from \eqref{bound r_1^c} that
\begin{equation}\label{lower bound on c}
    c>\frac{\log\log C_L(T)-2\log\log\log C_L(T)-\log 16}{\log(m+2)}.
\end{equation}
To derive an upper bound on $c$, note that \eqref{eqn upper bound c} gives
\[
r_2^c>\frac{a(T)}{13},
\]
and thus
\begin{equation}\label{upper bound on c}
    c<\frac{\log(13/a(T))}{\log(1/r_2)}<\log\left(\frac{13}{a(T)}\right)\frac{1+r_2}{2(1-r_2)}<\log\left(\frac{13}{a(T)}\right)\frac{1}{2}e^{(1+2\vartheta)/a(T)+(4+2\log 2)}
\end{equation}
by \eqref{bound (1+r_2)/(1-r_2)}. Combining \eqref{lower bound on c} and \eqref{upper bound on c}, we arrive at
\[
\frac{1+2\vartheta}{a(T)}+(4+\log 2)+\log\log\left(\frac{13}{a(T)}\right)>\log\left(\frac{\log\log C_L(T)}{\log(m+2)}\right)+o(1)
\]
as $C_L(T)\to \infty$, so that
\[
a(T)<\dfrac{(1+2\vartheta)}{\log\left(\frac{\log\log C_L(T)}{\log(m+2)}\right)-\log\log\log\left(\frac{\log\log C_L(T)}{\log(m+2)}\right)-(4+\log 2+o(1))}
\]
as $\frac{\log\log C_L(T)}{\log(m+2)}\to \infty$. Since this is weaker than the conclusion \eqref{bound a(T) case 1} of Case 1, the proof of Theorem~\ref{thm HH} is complete.

We now turn to the proof of Theorem~\ref{thm Siegel}.

\section{Proof of Theorem~\ref{thm Siegel}}

We need the following lemma of Siegel (see \cite[pp. 419]{Sie45}):
\begin{lemma}\label{lemma Siegel}
    Let $\lambda>0$, $0<\xi<1$, $0<M_0<M$. Suppose that $f(z)$ is analytic in the rectangle $\{z=x+iy: 0\leq x\leq 1, -\frac{1}{2\lambda}\leq y\leq \frac{1}{2\lambda}\}$ such that $\Re f(z)\leq M$ throughout this rectangle and $|f(z)|\leq M_0$ on the right edge $x=1$. Then
    \[
    \dfrac{\log \left|\dfrac{2M}{f(\xi)}-1\right|}{\log \left|\dfrac{2M}{M_0}-1\right|}\geq \frac{\sinh(\pi\lambda\xi)}{\sinh(\pi\lambda)}.
    \]
\end{lemma}

We now proceed to the proof of Theorem~\ref{thm Siegel}. Throughout it is assumed that $C_L(T)^{1/m}$ is large. Suppose that $L(s)$ has no non-trivial zero with $|\gamma-T|\leq b(T)$ where
\[
\frac{1}{\log\log\log C_L(T)^{1/m}}\leq b(T)\leq \frac{1}{10}.
\]
Set 
\[
\sigma_0=1+\vartheta+2b(T), \quad s_0=\sigma_0+iT, \quad\lambda=\frac{\sigma_0+3/2}{2b(T)}, \quad z=\frac{s+3/2-iT}{\sigma_0+3/2}.
\]
Further let $d=\mathrm{ord}_{s=0}L(s)$ and
\[
J_1=\{1\leq j\leq m: \mu_j\neq 0, \: -1<\Re(\mu_j)\leq 3/2\},
\]
\[
J_2=\{1\leq j\leq m: \mu_j\neq 0, \: -1<\Re(\mu_j)\leq -1/2\}.
\]
We work with the function
\[
f(z)=\log \frac{g(s)}{g(s_0)} \quad \text{where}\:\: g(s)=\frac{L(s)(s-1)^{r_L(T)}}{s^d \prod_{j\in J_1}(s+\mu_j) \prod_{j\in J_2}(s+2+\mu_j)}
\]
(here, as in \eqref{def r_L(T)}, $r_L(T)=r_L$ if $|T|\leq 1$ and 0 otherwise). In particular, $f(z)$ is holomorphic in the rectangle 
\[
Q=\left\{z: 0\leq x\leq 1, \:-\frac{1}{2\lambda}\leq y\leq \frac{1}{2\lambda}\right\}=\{s: -3/2\leq \sigma\leq \sigma_0, \:|t-T|\leq b(T)\}.
\]

Let $z$ be a point on the right edge of $Q$. Following the proof of Lemma~\ref{lemma: log L(s_0)/L(s_1)}, we find that
\[
\left|\log \frac{L(s)}{L(s_0)}\right|\leq m\log\left(\frac{2b(T)}{2b(T)-|s-s_0|}\right)\leq m \log 2.
\]
Furthermore, note that
\[
\left|\frac{s_0-s}{s+\mu_j}\right|<\frac{b(T)}{\vartheta+2b(T)}\leq \frac{1}{2}
\]
for any $\mu_j$, and hence 
\begin{align*}
    &\left|\log \left(\left(\frac{s_0}{s}\right)^d \prod_{j\in J_1}\frac{s_0+\mu_j}{s+\mu_j} \prod_{j\in J_2}\frac{s_0+2+\mu_j}{s+2+\mu_j}\left(\frac{s-1}{s_0-1}\right)^{r_L(T)}\right)\right|\\
    &\hspace{1cm} \leq
    d\left|\log\left(1+\frac{s_0-s}{s}\right)\right|+\sum_{j\in J_1}\left|\log\left(1+\frac{s_0-s}{s+\mu_j}\right)\right|+\sum_{j\in J_2}\left|\log\left(1+\frac{s_0-s}{s+2+\mu_j}\right)\right|\\
    &\hspace{5cm}+m\left|\log\left(1+\frac{s-s_0}{s_0-1}\right)\right|\\
    &\hspace{1cm} <3m,
\end{align*}
where we used \eqref{|log(1+z)|} to bound the logarithms. We can therefore choose 
\begin{equation}\label{M_0 Siegel}
    M_0=4m
\end{equation}
in Lemma~\ref{lemma Siegel}.

Now let $z$ be an arbitrary point in $Q$. In a way similar to the proof of Lemma~\ref{claim: K(f)}, another application of Lemma~\ref{lemma: |L(s)| bound} yields 
\begin{align*}
    \Re \log g(s)<& m+\frac{(5/2-\sigma)}{2} \log C_L(T)+
    \begin{cases}
        r_L\log|s-4| & \text{if $|T|\leq 1$},\\
        r_L\log \left|\dfrac{s-4}{s-1}\right| & \text{if $|T|>1$}.
    \end{cases}\\
    \leq& 2\log C_L(T)+O(m).
\end{align*}
On the other hand, 
\[
\log \left|\frac{s_0^d}{(s_0-1)^{r_L(T)}} \prod_{j\in J_1}(s_0+\mu_j) \prod_{j\in J_2}(s_0+2+\mu_j)\right|<3\log C_L(T)
\]
and (as in \eqref{log L(s_1) upper bound})
\[
|\log L(s_0)|\leq \sum_{n=1}^\infty \frac{m\Lambda(n)n^\vartheta}{n^{1+\vartheta+2b(T)}\log n}=m\log \zeta(1+2b(T))
<m\log\log\log\log C_L(T)^{1/m}.
\]
Thus we may take 
\begin{equation}\label{M Siegel}
    M=6\log C_L(T),
\end{equation}
say.

Next, set 
\[
s_1=1-\overline{s_0}=1-\sigma_0+iT, \qquad \xi=\frac{s_1+3/2-iT}{\sigma_0+3/2}=\frac{5/2-\sigma_0}{\sigma_0+3/2}.
\]
In particular, we have $0<\xi<1$. Put $d'=\#\{1\leq j\leq m: \mu_j=0\}$ (observe that $0\leq d\leq d'$ since $L(1)\neq 0$ and $d'-d=r_L$). By the functional equation and \eqref{gamma relation},
\begin{align*}
    \Re f(\xi)=& \log \left|\frac{L(s_1)s_0^d (s_1-1)^{r_L(T)}\prod_{j\in J_1}(s_0+\mu_j)\prod_{j\in J_2}(s_0+2+\mu_j)}{L(s_0) s_1^d (s_0-1)^{r_L(T)} \prod_{j\in J_1}(s_1+\mu_j)\prod_{j\in J_2}(s_1+2+\mu_j)}\right|\\
    =&\left(\sigma_0-\frac{1}{2}\right)\log \frac{N_L}{\pi^m}+\sum_{j\in J_1\setminus J_2} \log \left|\frac{\Gamma(\frac{2+s_0+\mu_j}{2})}{\Gamma(\frac{2+s_1+\mu_j}{2})}\right|+\sum_{j\in J_2} \log \left|\frac{\Gamma(\frac{4+s_0+\mu_j}{2})}{\Gamma(\frac{4+s_1+\mu_j}{2})}\right|\\
    &+\sum_{\substack{\mu_j\neq 0\\j\not \in J_1}}\log\left|\frac{\Gamma(\frac{s_0+\mu_j}{2})}{\Gamma(\frac{s_1+\mu_j}{2})}\right|+d \log \left|\frac{\Gamma(\frac{2+s_0}{2})}{\Gamma(\frac{2+s_1}{2})}\right|+(d'-d) \log \left|\frac{\Gamma(\frac{s_0}{2})}{\Gamma(\frac{s_1}{2})}\right|+r_L(T)\log\left|\frac{s_1-1}{s_0-1}\right|.
\end{align*}
The last term vanishes if $|T|>1$. Otherwise, observe that $|s_0-1|=|s_1|$ and $|s_1-1|=|s_0|$, so the last three terms can be combined into
\begin{equation}\label{combine indices}
    d'\log \left|\frac{\Gamma(\frac{2+s_0}{2})}{\Gamma(\frac{2+s_1}{2})}\right|.
\end{equation}
Since $s_0-s_1=2\sigma_0-1$, by the mean-value theorem 
\[
\Re f(\xi)=\left(\sigma_0-\frac{1}{2}\right) \left(\log\frac{N_L}{\pi^m}+\sum_{j=1}^m \Re\frac{\Gamma'}{\Gamma}\left(\frac{\sigma^*_j+iT+\mu_j}{2}\right)\right)
\]
where $\sigma^*_j$ are real numbers lying in corresponding intervals (e.g., $2+\Re(s_1)\leq \sigma^*_j\leq 2+\sigma_0$ for $j\in J_1\setminus J_2$). The goal is to show that the above is $\gg \log C_L(T)$. For each $\mu_j\neq 0$, we check case by case that $\sigma_j^*+\Re(\mu_j)>1/4$, and hence by Stirling's formula \eqref{Stirling}
\[
\Re\frac{\Gamma'}{\Gamma}\left(\frac{\sigma^*_j+iT+\mu_j}{2}\right) \geq \log(3+|iT+\mu_j|)+O(1).
\]
It remains to treat the $d'$ indices $j$ with $\mu_j=0$. If $|T|\geq 1$ we have
\[
\Re\frac{\Gamma'}{\Gamma}\left(\frac{\sigma^*_j+iT}{2}\right) \geq \log(3+|T|)+O(1)-\frac{\sigma_j^*}{{\sigma_j^*}^2+T^2}\geq \log(3+|T|)+O(1).
\]
If $|T|\leq 1$, by \eqref{combine indices} each of these indices satisfies $\sigma^*_j\geq 2+\Re(s_1)=2-\vartheta-2b(T)>1/2$, so the same conclusion holds. It follows that
\begin{equation}\label{f(xi) Siegel}
    \Re f(\xi)\geq \left(\sigma_0-\frac{1}{2}\right)\log C_L(T)+O(m) > \frac{\log C_L(T)}{3}.
\end{equation}

Putting \eqref{M_0 Siegel}, \eqref{M Siegel} and \eqref{f(xi) Siegel} together, we arrive at
\[
\log \left|\frac{2M}{f(\xi)}-1\right|\leq \log \left|\frac{2M}{\Re f(\xi)}-1\right|< \log \left(\frac{2\cdot 6}{1/3}-1\right)=O(1),
\]
\[
\log \left|\frac{2M}{M_0}-1\right|=\log \left(\frac{2\cdot 6}{4}\log C_L(T)^{1/m}-1\right)=\log\log C_L(T)^{1/m}+O(1).
\]
According to Lemma~\ref{lemma Siegel}, we have $A\geq B$ where
\[
A=\log \left|\frac{2M}{f(\xi)}-1\right| \bigg / \log \left|\frac{2M}{M_0}-1\right| = O\left(\frac{1}{\log\log C_L(T)^{1/m}}\right),
\]
\[
B=\frac{\sinh(\pi\lambda\xi)}{\sinh(\pi\lambda)}=\frac{e^{\pi\lambda\xi}-e^{-\pi\lambda\xi}}{e^{\pi\lambda}-e^{-\pi\lambda}}=e^{\pi\lambda(\xi-1)}+O(e^{-\pi \lambda(\xi+1)})=e^{(1/2-\sigma_0)\pi/b(T)}+O(e^{-2\pi/b(T)}),
\]
so that
\[
b(T)\leq \frac{\pi}{2}\frac{1+2\vartheta+4b(T)}{\log\log\log C_L(T)^{1/m}+O(1)}.
\]
The stated estimate in Theorem~\ref{thm Siegel} follows readily.

\section{Acknowledgment}
The author would like to thank Neea Paloj\"{a}rvi and the anonymous referees for their valuable suggestions.

\printbibliography

\appendix
\section*{Appendix}
\begin{lemma}
    \begin{equation}\label{log deri zeta}
    -\frac{\zeta'}{\zeta}(\sigma)<\frac{1}{\sigma-1}, \qquad \sigma>1.
    \end{equation}
\end{lemma}
\begin{proof}
    We start by recalling the identity (see, e.g., \cite[(1.24)]{MV})
    \begin{equation}\label{zeta}
        \zeta(s)=\frac{s}{s-1}-s\int_1^\infty \{u\}u^{-s-1} \md u, \qquad \sigma>0
    \end{equation}
    where $\{u\}=u-\lfloor u \rfloor$ denotes the fractional part of $u$. Differentiating gives
    \[
    -\zeta'(s)=\frac{1}{(s-1)^2}+\int_1^\infty \{u\}u^{-s-1}\md u-s\int_1^\infty \{u\}u^{-s-1}\log u\md u,
    \]
    and so 
    \[
    -\dfrac{\zeta'}{\zeta}(\sigma)<\dfrac{\frac{1}{(\sigma-1)^2}+\int_1^\infty \{u\}u^{-\sigma-1}\md u}{\frac{\sigma}{\sigma-1}-\sigma\int_1^\infty \{u\}u^{-\sigma-1} \md u}.
    \]
    After a little rearranging, \eqref{log deri zeta} would follow if we can establish
    \begin{equation}\label{frac part intergal}
        \int_1^\infty \{u\}u^{-\sigma-1}\md u<\frac{1}{2\sigma-1}, \qquad \sigma>1.
    \end{equation}
    \footnote{In view of \eqref{zeta}, the bound \eqref{frac part intergal} is equivalent to
    \[
    \zeta(\sigma)>\frac{\sigma^2}{2\sigma^2-3\sigma+1}, \qquad \sigma>1.
    \]} The left-hand side equals
    \begin{align*}
        \int_1^\infty (u-\lfloor u\rfloor)u^{-\sigma-1}\md u =& \frac{1}{\sigma-1}-\sum_{n=1}^\infty n\int_{n}^{n+1}u^{-\sigma-1}\md u \\
        =&\frac{1}{\sigma-1}-\frac{1-2^{-\sigma}}{\sigma}-\sum_{n=2}^\infty \frac{n}{\sigma}\left(\frac{1}{n^{\sigma}}-\frac{1}{(n+1)^{\sigma}}\right).
    \end{align*}
    Using the inequality $(1+x)^\sigma\geq 1+\sigma x$ for $\sigma>1$, we see that the above sum is at least
    \[
    \sum_{n=2}^\infty \frac{n}{\sigma}\frac{\sigma}{n(n+1)^\sigma}=\sum_{n=3}^\infty \frac{1}{n^\sigma}>\int_{3}^\infty u^{-\sigma}\md u=\frac{3^{1-\sigma}}{\sigma-1}.
    \]
    Thus, in view of \eqref{frac part intergal}, it suffices to show that
    \[
    \frac{1}{\sigma-1}-\frac{1-2^{-\sigma}}{\sigma}-\frac{3^{1-\sigma}}{\sigma-1}<\frac{1}{2\sigma-1},
    \]
    which can easily be verified for all $\sigma>1$.
\end{proof}

\end{document}